\documentclass[11pt, a4paper]{article}
\usepackage[utf8]{inputenc}
\usepackage[T1]{fontenc}
\usepackage[english]{babel}
\usepackage{amsmath, amsthm, amssymb, amsfonts}
\usepackage{mathtools}
\usepackage{geometry}
\usepackage{xcolor}
\usepackage{hyperref}
\usepackage{url}
\usepackage{indentfirst}

\DeclareMathOperator{\Ima}{im} 
\newtheorem{theorem}{Theorem}[section]
\newtheorem{lemma}[theorem]{Lemma}
\newtheorem{proposition}[theorem]{Proposition}
\newtheorem{corollary}[theorem]{Corollary}
\theoremstyle{definition}

\newtheorem{example}[theorem]{Example}
\newtheorem{remark}[theorem]{Remark}

\title{On the Regularity and Clean Properties of Matrices over Dual Numbers}

\author{\small BERR\.{I}N \c SENT\" URK}

\begin{document}

\maketitle

\begin{abstract}
In this paper, we study regularity and clean-type decompositions for matrices over dual numbers. We characterize the von Neumann regularity of dual matrices by establishing a necessary and sufficient compatibility condition between their real and dual components. At the element level, we determine the conditions under which an individual dual matrix $\mathcal{M} = A + B\epsilon$ is regular, $\pi$-regular, nil-clean, or strongly nil-clean. Using the solvability of Sylvester matrix equations, we show that the matrix ring over dual numbers $M_n(\mathbb{D})$ is strongly $\pi$-regular and strongly clean. Moreover, $M_n(\mathbb{D})$ is $\pi$-regular, clean, $r$-clean, strongly $r$-clean, $\mathrm{NR}$-clean, and strongly $\mathrm{NR}$-clean. On the other hand, $M_n(\mathbb{D})$ fails to be regular, strongly regular, nil-clean, strongly nil-clean, and uniquely $\mathrm{NR}$-clean, although it inherits several strong decomposition properties from $M_n(\mathbb{R})$. Our characterizations at the element level show exactly when each property holds.

\vspace{0.3cm}
    
\end{abstract}

\section{Introduction}\label{sec1}

The ring of dual numbers, denoted by $\mathbb{D} = \{a + b\epsilon \mid a, b \in \mathbb{R}, \epsilon^2 = 0\}$, was introduced by Clifford \cite{Clifford1871} and study of its matrix representations has become an important topic in linear algebra, mechanics, kinematics, and robotics \cite{fis99}. Any $n \times n$ matrix over dual numbers, $\mathcal{M} \in M_n(\mathbb{D})$, can be uniquely expressed as a linear combination $\mathcal{M} = A + B\epsilon$, where $A$ and $B$ are any real matrices in $M_n(\mathbb{R})$.

In ring theory, structural properties that decompose elements into special components such as idempotents, units, and nilpotents play a central role. Throughout this paper, all rings are assumed to be associative with identity. A fundamental such property is von Neumann regularity, that is, an element $a$ in a ring $R$ is called \emph{von Neumann regular} (simply \emph{regular}) if  $a \in aRa$ \cite{Neumann1936}. Any element $a^{-}\in R$ satisfying $a= aa^{-}a$ is called an \emph{inner inverse} of $a$. A ring is called \emph{regular} if all of its elements are regular.  An element $a$ in a ring $R$ is called \emph{strongly regular} if $a \in Ra^2\cap a^2R$ \cite{diesl2013nil}. A ring $R$ is called a \emph{strongly regular ring} if every element in $R$ is strongly regular. Several equivalent conditions for an element to be strongly regular can be found in Nicholson \cite{Nicholson1999}. Moreover, relaxing the regularity condition to higher powers gives $\pi$-regularity; we say an element $a$ is \emph{$\pi$-regular} if $a^n$ is regular for some integer $n\geq 1$ \cite{McCoy1939}. Strengthening $\pi$-regularity by requiring commutativity yields strong $\pi$-regularity; an element $a$ in a ring is called \emph{strongly $\pi$-regular} if $a^n \in Ra^{n+1}\cap a^{n+1}R$ for some integer $n\geq 1$ \cite{diesl2013nil}.

On the clean side, a ring is called \textit{clean}, introduced by Nicholson \cite{Nicholson1977}, if every element can be written as the sum of a unit and an idempotent. If these two components commute, the ring is called \textit{strongly clean} \cite{Nicholson1999}. Similarly, Diesl \cite{diesl2013nil} introduced the concept of \textit{nil-clean rings}, where elements can be decomposed as the sum of an idempotent and a nilpotent. If these two components commute, the element is called \textit{strongly nil-clean}. An element $a$ in a ring is called \textit{r-clean} (or \textit{regular-clean}) if it can be written as the sum of a regular element and an idempotent \cite{Ashrafi2011rcleanR}, and \textit{strongly r-clean} if these two components commute \cite{sharma2018stronglyrclean}. Finally, a ring is called \textit{$\mathrm{NR}$-clean} if every element can be written as the sum of a regular element and a nilpotent \cite{Khashan2016}, and  \textit{strongly $\mathrm{NR}$-clean} if these two components additionally commute \cite{Ibraheem2021}.

We note that the ring of dual numbers is the trivial extension $\mathbb{D}\cong \mathbb{R} \propto \mathbb{R}$ of $\mathbb{R}$ by itself, and consequently $M_n(\mathbb{D})\cong M_n(\mathbb{R}) \propto M_n(\mathbb{R})$. In particular, \cite[Corollary~3.23]{diesl2013nil} shows that $M_n(\mathbb{D})$ is nil-clean (respectively, strongly nil-clean) exactly when $M_n(\mathbb{R})$ is. However, these results work at the level of rings and give no information about the dual part $B$ of a particular matrix. In this paper, we also work at the element level. We give necessary and sufficient conditions on the real and dual parts of 
$\mathcal{M}=A + B\epsilon$ to possess each property, such as the compatibility condition $(I-AA^{-})B(I-A^{-}A)=0$ for regularity. Moreover, as outlined below, we build the strongly clean and strongly $\pi$-regular decompositions of $M_n(\mathbb{D})$ explicitly, and their dual parts are uniquely determined.

Although the real matrix ring $M_n(\mathbb{R})$ is well known to be both von Neumann regular \cite{Goodearl1991} and strongly clean \cite{Nicholson1999}, these properties behave very differently under the extension to $M_n(\mathbb{D})$. Due to the existence of nonzero nilpotent elements arising from $\epsilon$, the ring $M_n(\mathbb{D})$ is not regular. However, certain global properties lift directly to $M_n(\mathbb{D})$. For example, $M_n(\mathbb{D})$ remains strongly $\pi$-regular, strongly clean, and strongly $\mathrm{NR}$-clean. At the level of individual elements, a dual matrix 
$\mathcal{M} = A + B\epsilon$ need not be regular. Its regularity is characterized by a compatibility condition between the real 
part $A$ and the dual part $B$. A similar compatibility condition characterizes $\pi$-regularity at each fixed power, but it is always satisfied for 
a sufficiently large power. For strong nil-cleanness, no such condition on $B$ is needed; $\mathcal{M}$ is 
strongly nil-clean precisely when its real part $A$ is.

The main idea behind our constructions is the connection between Fitting's Lemma and strongly clean decompositions, established by Nicholson \cite{Nicholson1999} and developed further by Diesl \cite{diesl2013nil} in the setting of nil clean rings. The characterization of strongly $\pi$-regular elements due to Diesl is the form we use in this paper. In module theory, Fitting's Lemma states that an endomorphism of a module of finite length decomposes it into a direct sum of two invariant submodules, on which the endomorphism is nilpotent and invertible, respectively. In our setting, for a real matrix $A$ in $M_n(\mathbb{R})$, this decomposition of $\mathbb{R}^n$ yields an 
idempotent $E$ commuting with $A$ that splits $A$ into two diagonal blocks with no 
common eigenvalues, one nilpotent and one invertible. When we lift such a decomposition to a dual matrix $\mathcal{M} = A + B\epsilon$, 
the dual part of the lifted idempotent satisfies Sylvester matrix equations. Since 
the diagonal blocks of $A$ have no eigenvalues in common, these equations have 
unique solutions. This is the primary step in proving our main result that $M_n(\mathbb{D})$ is strongly 
$\pi$-regular, and consequently strongly clean.

The paper is organized as follows. In Section 2, we introduce the canonical projection map from $M_n(\mathbb{D})$ to $M_n(\mathbb{R})$. We then characterize invertible, idempotent, and nilpotent dual matrices. Section 3 focuses on regular, strongly regular, $\pi$-regular and strongly $\pi$-regular properties. In Section 4, we examine the clean, strongly clean, nil-clean, strongly nil-clean, $r$-clean, strongly $r$-clean, $\mathrm{NR}$-clean  and strongly $\mathrm{NR}$-clean  properties. In section 5, we provide a summary comparing the global and local behaviors of these structures.

\section{Characterization of Some Special Matrices over Dual Numbers}

Let $\mathbb{D}=\{a+b\epsilon \mid a,b \in \mathbb{R}, \epsilon^2=0\}$ be the ring of dual numbers. We denote the ring of $n \times n$ matrices over $\mathbb{D}$ by $M_n(\mathbb{D})$.
Any matrix $\mathcal{M} \in M_n(\mathbb{D})$ can be uniquely expressed as $\mathcal{M}=A+B\epsilon$, where $A,B \in M_n(\mathbb{R})$.

We define the \textit{real projection map} $\pi: M_n(\mathbb{D}) \to M_n(\mathbb{R})$ by $\pi(A+B\epsilon)=A$.
\begin{lemma}\label{lemma:proj}
	The map $\pi$ is a surjective ring homomorphism. That is, for any $\mathcal{M}, \mathcal{N} \in M_n(\mathbb{D})$:
	\begin{itemize}
		\item[$(i)$] $\pi(\mathcal{M}+\mathcal{N})=\pi(\mathcal{M})+\pi(\mathcal{N})$,
		\item[$(ii)$] $\pi(\mathcal{M}\mathcal{N})=\pi(\mathcal{M})\pi(\mathcal{N})$.
	\end{itemize}
\end{lemma}
\begin{proof}
 The proof is straightforward from the definition of matrix operations in $\mathbb{D}$.
\end{proof}
\subsection{Invertible Matrices}
We characterize the units (invertible elements) of the ring $M_n(\mathbb{D})$.

\begin{theorem} \label{thm:invertible}
	Let $\mathcal{M} = A+B\epsilon \in M_n(\mathbb{D})$. Then $\mathcal{M}$ is invertible if and only if its real projection $A$ is invertible in $M_n(\mathbb{R})$ (i.e., $\det(A) \neq 0$).
\end{theorem}

\begin{proof}
	($\Rightarrow$) Suppose $\mathcal{M}$ is invertible. Then there exists $\mathcal{X}$ such that $\mathcal{M}\mathcal{X}=I$. Applying the projection map $\pi$, we have $\pi(\mathcal{M})\pi(\mathcal{X})=\pi(I)$, which implies $A\pi(\mathcal{X})=I$. Thus $A$ is invertible.
	
	($\Leftarrow$) Suppose $A$ is invertible. Consider the matrix $ \mathcal{X}=A^{-1}-A^{-1}BA^{-1} \epsilon $.
	Then
	\begin{align*}
	\mathcal{M}\mathcal{X}&=(A+B\epsilon)(A^{-1}-A^{-1}BA^{-1}\epsilon) \\
	&= AA^{-1}-A(A^{-1}BA^{-1})\epsilon+BA^{-1}\epsilon-B(A^{-1}BA^{-1})\epsilon^2 \\
	&= I-BA^{-1}\epsilon+BA^{-1}\epsilon-0 \\
	&= I.
	\end{align*}
	Similarly, $\mathcal{X}\mathcal{M} = I$. Thus $\mathcal{M}$ is invertible.
\end{proof}

\subsection{Idempotent Matrices}
In this section, we characterize idempotent matrices in $M_n(\mathbb{D})$. A matrix $\mathcal{M}$ is called idempotent if $\mathcal{M}^2=\mathcal{M}$.

\begin{theorem} \label{thm:idempotent}
    Let $\mathcal{M}=A+B\epsilon \in M_n(\mathbb{D})$. Then $\mathcal{M}$ is idempotent if and only if
    \begin{enumerate}
        \item[$(i)$] $A$ is an idempotent matrix in $M_n(\mathbb{R})$, and
        \item[$(ii)$] $B$ satisfies the condition $AB+BA=B$.
    \end{enumerate}
\end{theorem}

\begin{proof}
    Assume $\mathcal{M}^2=\mathcal{M}$. Then $$ (A+B\epsilon)^2 = A^2+(AB+BA)\epsilon=A+B\epsilon.$$ Therefore, we obtain $A^2=A$ and $AB+BA=B$.
    Conversely, if these conditions hold, it is easy to verify that $\mathcal{M}^2 = \mathcal{M}$.
\end{proof}

\begin{remark}\label{remark:off-diagonal}
Suppose that the conditions $(i)$ and $(ii)$ hold:  $A$ is an idempotent matrix in $M_n(\mathbb{R})$ and $AB+BA=B$. 
   Then $A$ is diagonalizable with eigenvalues in $\{0,1\}$. Let $r=\mathrm{rank}\, A$, and choose $P\in GL_n(\mathbb{R})$ with $P^{-1}AP=\begin{pmatrix} I_r & 0 \\ 0 & 0 \end{pmatrix}$. Let $A'=P^{-1}AP$ and $B'=P^{-1}BP$, and write $B'=\begin{pmatrix} B_{11} & B_{12} \\ B_{21} & B_{22} \end{pmatrix}$ with block sizes $r$ and $n-r$.
 If we conjugate $AB+BA=B$ by $P$, we get $A'B'+B'A'=B'$:  
$$ \begin{pmatrix} I_r & 0 \\ 0 & 0 \end{pmatrix} \begin{pmatrix} B_{11} & B_{12} \\ B_{21} & B_{22} \end{pmatrix}+ \begin{pmatrix} B_{11} & B_{12} \\ B_{21} & B_{22} \end{pmatrix} \begin{pmatrix} I_r & 0 \\ 0 & 0 \end{pmatrix} = \begin{pmatrix} 2B_{11} & B_{12} \\  B_{21} & 0 \end{pmatrix}= \begin{pmatrix} B_{11} & B_{12} \\ B_{21} & B_{22} \end{pmatrix}. $$
Hence $ B_{11}=0 $, and $B_{22}=0$. Thus the dual component $B'$ lies in the off-diagonal blocks:
$ B'=\begin{pmatrix} 0 & B_{12} \\ B_{21} & 0 \end{pmatrix}. $ Conversely, for any $B_{12}\in M_{r\times (n-r)}(\mathbb{R})$ and $B_{21}\in M_{(n-r)\times r}(\mathbb{R})$, the matrix $A'+B'\epsilon$ with $B'$ as above is idempotent since $A'^2=A'$ and $A'B'+B'A'=B'$ hold by direct computation.
Consequently, whenever a statement about an idempotent $\mathcal{M}=A+B\epsilon $ is invariant under conjugation by $ GL_n(\mathbb{R})$, we may assume, without loss of generality, that $A= \begin{pmatrix} I_r & 0 \\ 0 & 0 \end{pmatrix}$ and $B=\begin{pmatrix} 0 & B_{12} \\ B_{21} & 0 \end{pmatrix}$ with $r=\mathrm{rank}\, A$.
\end{remark}

\subsection{Nilpotent Matrices}
Now, we investigate nilpotent elements. A matrix $\mathcal{M}$ is nilpotent if $\mathcal{M}^k = 0$ for some integer $k \ge 1$.

\begin{theorem} \label{thm:nilpotent}
    A dual matrix $\mathcal{M}=A+B\epsilon$ is nilpotent if and only if its real part $A$ is nilpotent.
\end{theorem}

\begin{proof}
    ($\Rightarrow$) Suppose $\mathcal{M}$ is nilpotent. Then $\mathcal{M}^k=0$ for some integer $k \ge 1$. Applying the homomorphism $\pi$, we get $\pi(\mathcal{M}^k)=(\pi(\mathcal{M}))^k=A^k=0$. Thus, $A$ is nilpotent.

    ($\Leftarrow$) Suppose $A$ is nilpotent with index $k$. Consider $\mathcal{M}^{2k}$: 
    $$ \mathcal{M}^{2k}=(A+B\epsilon)^{2k}=\underbrace{A^{2k}}_{0}+\left( \sum_{j=0}^{2k-1}A^jBA^{(2k-1)-j} \right)\epsilon .$$

    Since $\epsilon^2=0$, this expansion will contain terms with $A$ with power $ \geq k. $ Thus $\mathcal{M}$ is nilpotent.
\end{proof}

\subsection{Sylvester Equation}
In subsequent proofs, we will use Sylvester's well-known theorem \cite{Sylvester84}, \cite[Theorem~4.6]{horn94}, classically stated for matrices with complex entries, adapted here to the real entries.  
\begin{lemma}\label{lemma:Sylvester}
  Let $L\in M_a(\mathbb{R}) ,R \in M_b(\mathbb{R})$ be given. The Sylvester Equation  $LX-XR=T$ has a unique solution $X \in M_{a,b}(\mathbb{R})$ for each $T\in   M_{a,b}(\mathbb{R})$ if and only if $L$ and $R$ have no eigenvalues in common.
\end{lemma}

\section{The Regularity Properties}

In this section, we will characterize the regular and $\pi$-regular elements of $M_n(\mathbb{D})$. We will show that although $M_n(\mathbb{D})$ is not regular, it is strongly 
$\pi$-regular.

\subsection{Regular and Strongly Regular Properties}
 Recall that an element $a$ in a ring $R$ is called  \emph{regular} if there exists $x\in R$ such that $a = axa$. A ring $R$ is \emph{regular} if every element of $R$ is regular.
\begin{proposition}\label{prop:regularCondition}
	A dual matrix $\mathcal{M} = A + B\epsilon$ is regular if and only if its real part $A$ is regular in $M_n(\mathbb{R})$ and $(I-AA^{-})B(I-A^{-}A)=0$ for some inner inverse $A^{-}$ of $A$.
\end{proposition}
\begin{proof}
    ($\Rightarrow$) Assume $\mathcal{M} = A + B\epsilon$ is regular. By definition there exists $\mathcal{X}=X+Y\epsilon$ such that $\mathcal{M}=\mathcal{M}\mathcal{X}\mathcal{M}$. Then
    $$A + B\epsilon=(A + B\epsilon)(X+Y\epsilon)(A + B\epsilon)=AXA +(AXB+AYA+BXA)\epsilon. $$
    Thus, $A=AXA$, so $A$ is regular and $X$ is an inner inverse of $A$; write $X=A^{-}$. We also have
$B=AXB+AYA+BXA$, where $X=A^{-}$. Note that $A=AXA$ implies that $(I-AX)AXB=0$ and $BXA(I-XA)=0$. Hence
\begin{align*}
(I-AX)B(I-XA)&= (I-AX)(AXB+AYA+BXA)(I-XA)\\
             &=0+(I-AX)AYA(I-XA)+0 \\
             &=(AYA-AXAYA)(I-XA)\\
             &=(AYA-AYA)(I-XA)=0.
\end{align*}
Therefore, if $\mathcal{M}$ is regular, then $A$ is regular and $(I-AX)B(I-XA)=0$ for the inner inverse $A^{-}=X$.

    ($\Leftarrow$) Assume $A$ is regular in $M_n(\mathbb{R})$ and there exists an inner inverse $A^{-}$ of $A$ such that the condition $(I-AA^{-})B(I-A^{-}A)=0$ holds. We fix this inner inverse $A^-$. We need to show that there exists an inner inverse $\mathcal{X}=X+Y\epsilon$ in $M_n(\mathbb{D})$ satisfying $\mathcal{M}=\mathcal{M}\mathcal{X}\mathcal{M}$, that is,
    $$A + B\epsilon=(A + B\epsilon)(X+Y\epsilon)(A + B\epsilon)=AXA +(AXB+AYA+BXA)\epsilon, $$ which is given by the equations $$A=AXA \mbox{ and } 
    AYA=B-AXB-BXA.$$ Setting  $X=A^{-}$ satisfies the first equation since $A^{-}$ is an inner inverse of $A$. Our aim is to show that the condition $(I-AA^{-})B(I-A^{-}A)=0$ implies that there exists $Y$  such that $AYA=B-AA^{-}B-BA^{-}A $. Expanding the condition, we get $$(B-AA^{-}B)(I-A^{-}A)=B-AA^{-}B-BA^{-}A+AA^{-}BA^{-}A=0.$$
   Then, $B-AA^{-}B-BA^{-}A=-AA^{-}BA^{-}A$, so $AYA=-AA^{-}BA^{-}A$. Therefore, it suffices to choose $Y=-A^{-}BA^{-}$. Hence there exists a dual matrix $\mathcal{X}=A^{-}-(A^{-}BA^{-})\epsilon$, which is an inner inverse of $\mathcal{M}$.
\end{proof}
\begin{example}\label{ex:regular} Let $A = \begin{pmatrix} 1 & 0 \\ 0 & 0 \end{pmatrix}$. Since $A$ is an idempotent, we may take $A^{-}=A$ and
$I - AA^{-} = I - A^{-}A = \begin{pmatrix} 0 & 0 \\ 0 & 1 \end{pmatrix}$. For  $B = \begin{pmatrix} 1 & 2 \\ 3 & 0 \end{pmatrix}$, we have $(I-AA^{-})B(I-A^{-}A)=0$, so $\mathcal{M}=A+ B\epsilon$ is regular by Proposition \ref{prop:regularCondition}. 
Indeed, $\mathcal{X}=A^{-}-(A^{-}BA^{-})\epsilon=\begin{pmatrix} 1 - \epsilon & 0 \\ 0 & 0 \end{pmatrix}$ satisfies $\mathcal{M}\mathcal{X}\mathcal{M} = \mathcal{M}$. On the other hand, for $B' = \begin{pmatrix} 1 & 2 \\ 3 & 1 \end{pmatrix}$ 
we have $(I - AA^{-})B'(I - A^{-}A) = \begin{pmatrix} 0 & 0 \\ 0 & 1 \end{pmatrix} 
\neq 0$, so $A + B'\epsilon$ is not regular. Note that only the $(2,2)$ entry of $B$ matters here. It is the part of $B$ 
that survives the projections in the compatibility condition, namely the 
restriction of $B$ to $\ker A$ followed by the projection along 
$\operatorname{im} A$.
\end{example}
\begin{theorem}\label{thmm:notregular} The matrix ring over dual numbers $M_n(\mathbb{D})$ is not regular.
\end{theorem}
	
\begin{proof} 
Consider the dual matrix $\mathcal{M} = A + B\epsilon$ with $A=0$ and $B\neq 0$. By Proposition \ref{prop:regularCondition}, $\mathcal{M}$ is regular if and only if $A$ is regular and $(I-AA^{-})B(I-A^{-}A)=0$ for some inner inverse $A^{-}$.  The zero matrix is regular and since $A=0$,  we have $AA^-=0$ and $A^-A=0$ for any choice of inner inverse $A^-$. The compatibility condition becomes $(I-0)B(I-0)=B\neq 0$. Hence $\mathcal{M}$ is not regular. Since  $M_n(\mathbb{D})$ contains a non-regular element, it is not a regular ring.
\end{proof}

Recall that an element $a$ in a ring $R$ is called \emph{strongly regular} if $a^2x=a=ya^2$ for some $x, y \in R$ \cite{Azumaya1954}. Equivalently, $a$ is strongly regular if $ aba=a$ and $ab=ba$ for some $b\in R$ \cite{Nicholson1999}. A ring $R$ is called  \emph{strongly regular} if every element in $R$ is strongly regular.

\begin{corollary}\label{corollary:notstrregular}
	The matrix ring over dual numbers $M_n(\mathbb{D})$ is not strongly regular.
\end{corollary}
\begin{proof} Obviously, every strongly regular ring is regular. Since $M_n(\mathbb{D})$ is not regular by Theorem \ref{thmm:notregular}, it is not strongly regular.
\end{proof}

\subsection{$\pi$-Regular and Strongly $\pi$-Regular Properties}

Recall that an element $a$ in a ring $R$ is called \emph{$\pi$-regular} if there exists an element $x\in R$ such that $a^n = a^nxa^n$ for some integer $n\geq 1$ \cite{McCoy1939}. We first characterize the $\pi$-regular elements of 
$M_n(\mathbb{D})$ in terms of their real and dual parts.
\begin{proposition}\label{prop:piregularCondition}
 A dual matrix $\mathcal{M} = A + B\epsilon$ is $\pi$-regular if and only if there exists an integer $ k \geq 1$ such that $A^k$ is regular in $M_n(\mathbb{R})$ and $(I-A^k(A^k)^{-})C_k(I-(A^k)^{-}A^k)=0$ for some inner inverse $(A^k)^{-}$ of $A^k$, where $ C_k:=\sum_{i=0}^{k-1}A^iBA^{(k-1)-i}$. 
\end{proposition}
\begin{proof}
($\Rightarrow$) Assume $\mathcal{M} = A + B\epsilon$ is $\pi$-regular. By definition there exist an integer $k\geq 1$ and $\mathcal{X}_k=X_k+Y_k\epsilon$ such that $\mathcal{M}^k=\mathcal{M}^k\mathcal{X}_k\mathcal{M}^k$. Define $ C_k:=\sum_{i=0}^{k-1}A^iBA^{(k-1)-i} \in M_n(\mathbb{R})$ so that $\mathcal{M}^k=A^k+C_k\epsilon$. Then
\begin{align*}A^k + C_k\epsilon &=(A^k + C_k\epsilon)(X_k+Y_k\epsilon)(A^k + C_k\epsilon)\\
                                &=A^kX_kA^k +(A^kX_kC_k+A^kY_kA^k+C_kX_kA^k)\epsilon. 
\end{align*}
    Thus, $A^k=A^kX_kA^k$, so  $A^k$ is regular and $X_k$ is an inner inverse of $A^k$; write $X_k=(A^k)^{-}$. We also have
$C_k=A^kX_kC_k+A^kY_kA^k+C_kX_kA^k$. Note that $A^k=A^kX_kA^k$ implies that $(I-A^kX_k)A^kX_kC_k=0$ and $C_kX_kA^k(I-X_kA^k)=0$. Hence
\begin{align*}
(I-A^kX_k)C_k(I-X_kA^k)&= (I-A^kX_k)(A^kX_kC_k+A^kY_kA^k+C_kX_kA^k)(I-X_kA^k)\\
 &=0+(I-A^kX_k)A^kY_kA^k(I-X_kA^k)+0 \\
  &=(A^kY_kA^k-A^kX_kA^kY_kA^k)(I-X_kA^k)\\
  &=(A^kY_kA^k-A^kY_kA^k)(I-X_kA^k)=0.
\end{align*}
Therefore, if $\mathcal{M}$ is $\pi$-regular, then for some integer $k\geq 1$, $A^k$ is regular and $(I-A^k(A^k)^{-})C_k(I-(A^k)^{-}A^k)=0$ for the inner inverse $(A^k)^-=X_k$.

 ($\Leftarrow$) Assume there exists an integer $k\geq 1$ such that $A^k$ is regular in $M_n(\mathbb{R})$ and $(I-A^k(A^k)^{-})C_k(I-(A^k)^{-}A^k)=0$ for some inner inverse $(A^k)^{-}$ of $A^k$, where $ C_k=\sum_{i=0}^{k-1}A^iBA^{(k-1)-i} \in M_n(\mathbb{R})$. We fix such an inner inverse $(A^k)^{-}$.  We need to show that there exists  $\mathcal{X}_k=X_k+Y_k\epsilon$ in $M_n(\mathbb{D})$ satisfying $\mathcal{M}^k=\mathcal{M}^k\mathcal{X}_k\mathcal{M}^k$, that is,
\begin{align*} A^k + C_k\epsilon & =(A^k + C_k\epsilon)(X_k+Y_k\epsilon)(A^k + C_k\epsilon)\\
                                 & =A^kX_kA^k +(A^kX_kC_k+A^kY_kA^k+C_kX_kA^k)\epsilon, 
\end{align*}                                  
which is given by the equations $$A^k=A^kX_kA^k \mbox{ and } A^kY_kA^k=C_k-A^kX_kC_k-C_kX_kA^k.$$ Setting $X_k=(A^k)^{-}$ satisfies the first equation since  $(A^k)^-$ is an inner inverse of $A^k$. Our aim is to show that the condition implies that there exists $Y_k$  such that $A^kY_kA^k=C_k-A^k(A^k)^{-}C_k-C_k(A^k)^{-}A^k $.

   Expanding the condition $(I-A^k(A^k)^{-})C_k(I-(A^k)^{-}A^k)=0$, we get 
\begin{align*}   (C_k-A^k(A^k)^{-}C_k)(I-(A^k)^{-}A^k)&=C_k-A^k(A^k)^{-}C_k\\
                                                      &\quad -C_k(A^k)^{-}A^k+A^k(A^k)^{-}C_k(A^k)^{-}A^k\\
                                                      &=0.
\end{align*}
   Then, $C_k-A^k(A^k)^{-}C_k-C_k(A^k)^{-}A^k=-A^k(A^k)^{-}C_k(A^k)^{-}A^k$, so $A^kY_kA^k=-A^k(A^k)^{-}C_k(A^k)^{-}A^k$. Therefore, it suffices to choose $Y_k=-(A^k)^{-}C_k(A^k)^{-}$. Consequently, there exists a dual matrix $$\mathcal{X}_k=(A^k)^{-}-((A^k)^{-}C_k(A^k)^{-})\epsilon,$$ which is an inner inverse of $\mathcal{M}^k$, so $\mathcal{M}^k$ is regular and hence $\mathcal{M}$ is 
$\pi$-regular.
\end{proof}

\begin{remark} Setting $k=1$ yields $C_1=B$, which recovers Proposition \ref{prop:regularCondition}: a dual matrix $\mathcal{M}=A+B\epsilon$ is regular if and only if $A$ is regular and $(I-AA^{-})B(I-A^{-}A)=0$ for some inner inverse $A^{-}$. The counterexample $A=0$ and $B\neq 0$ shows that $M_n(\mathbb{D})$ fails to be regular at $k=1$. On the other hand, this is no longer a counterexample for $k\geq 2$. When $A=0$, every term in 
$ C_k=\sum_{i=0}^{k-1}A^iBA^{(k-1)-i}$ contains at least one factor of $A$, which forces $C_k=0$. Then $\mathcal{M}^k=A^k+C_k\epsilon=0$, and we know that the zero matrix is trivially regular. Therefore, the obstacle preventing regularity at $k=1$ disappears under higher powers $k\geq 2$.

Also note that in Example \ref{ex:regular}, $A + B'\epsilon=\begin{pmatrix} 1 & 0 \\ 0 & 0 \end{pmatrix}+ \begin{pmatrix} 1 & 2 \\ 3 & 1 \end{pmatrix}\epsilon$ is not regular but is $\pi$-regular. For $k=2$, we have $A^2 = A$ and 
$C_2= AB' + B'A = \begin{pmatrix} 2 & 2 \\ 3 & 0 \end{pmatrix}$. Then 
$(I - A^2(A^2)^{-})C_2(I - (A^2)^{-}A^2) = 
\begin{pmatrix} 0 & 0 \\ 0 & 1 \end{pmatrix} C_2 
\begin{pmatrix} 0 & 0 \\ 0 & 1 \end{pmatrix} = 0$, 
since the $(2,2)$ entry of $C_2$ is zero. Hence the condition of 
Proposition~\ref{prop:piregularCondition} holds.
\end{remark}

 Recall that an element $a$ in a ring $R$ is called \emph{strongly $\pi$-regular} if $a^n \in Ra^{n+1}\cap a^{n+1}R$ for some integer $n\geq 1$ \cite[Definition~2.2]{diesl2013nil}. Equivalently, $a$ is strongly $\pi$-regular if there exist an integer $n\geq 1$ and an element $x \in R$ such that $a^n = a^{n+1}x$ and $ax=xa$ \cite[Theorem~3]{Azumaya1954}. Alternatively, by \cite[Proposition~2.5]{diesl2013nil}, $a$ is 
strongly $\pi$-regular if and only if it can be written as $a=e+u$, $ea=ae$ and $eae$ is nilpotent, where $e$ is an idempotent and $u$ is a unit. We use this form in the following proofs. A ring $R$ is \emph{strongly $\pi$-regular} if every element of $R$ is strongly $\pi$-regular. 

\begin{lemma}\label{thm:realstronglyclean}
Every matrix $A \in M_n(\mathbb{R})$ can be decomposed as $A = E + U$, where $E$ is an 
idempotent, $U$ is a unit, $EU = UE$, and $AE$ is nilpotent in $M_n(\mathbb{R})$. In a 
basis adapted to the decomposition $\mathbb{R}^n=\Ima{(E)} \oplus \ker(E)$, 
the matrix $A$ is block diagonal,
$$A= \begin{pmatrix} A_{11} & 0 \\ 0 & A_{22} \end{pmatrix},$$
where $A_{11}$ is nilpotent, $A_{22}$ is invertible, and hence $A_{11}$ and 
$A_{22}$ have no eigenvalues in common.
\end{lemma}

\begin{proof}
By Fitting's lemma, $\mathbb{R}^n = \ker(A^n) \oplus \Ima(A^n)$, where 
$A(\ker(A^n)) \subseteq \ker(A^n)$ and $A(\Ima(A^n)) \subseteq \Ima(A^n)$. 
Moreover, $A$ is nilpotent on $\ker(A^n)$ and invertible on $\Ima(A^n)$.
Write each element $v\in \mathbb{R}^n$ uniquely as $v=v_1+v_2$ with $v_1\in \ker(A^n)$ and $v_2\in \Ima(A^n)$, and define $E$ in $M_n(\mathbb{R})$ by $Ev=v_1$. Then
$\Ima(E) = \ker(A^n)$, $\ker(E) = \Ima(A^n) $, and $E^2 = E$. Since both summands are 
$A$-invariant, $EA = AE$. In an adapted basis, 
$$E = \begin{pmatrix} I_r & 0 \\ 0 & 0 \end{pmatrix} \mbox{ and }
A = \begin{pmatrix} A_{11} & 0 \\ 0 & A_{22} \end{pmatrix},$$
where $r = \dim\ker(A^n)$. Since $A_{11}$ and $A_{22}$ are the matrices of the restrictions of $A$ to $\ker(A^n)$ and $\Ima(A^n)$, respectively, $A_{11}$ is nilpotent, and $A_{22}$ is invertible. In 
particular, $AE = \begin{pmatrix} A_{11} & 0 \\ 0 & 0 \end{pmatrix}$ is nilpotent. 
Set
$$U = A - E = \begin{pmatrix} A_{11} - I_r & 0 \\ 0 & A_{22} \end{pmatrix}.$$
Since $A_{11}$ is nilpotent, $A_{11} - I_r$ is invertible. Because $A_{22}$ is also 
invertible, $U$ is a unit, and $EU = UE$. Finally, the only eigenvalue of $A_{11}$ 
is zero, while zero is not an eigenvalue of $A_{22}$. Hence $A_{11}$ and $A_{22}$ 
have no common eigenvalues.
\end{proof}

\begin{corollary}\label{cor:realstrpireg}
The matrix ring $M_n(\mathbb{R})$ is strongly $\pi$-regular.
\end{corollary}
\begin{proof}
Let $A \in M_n(\mathbb{R})$. By Lemma~\ref{thm:realstronglyclean}, $A = E + U$, where $E$ 
is an idempotent, $U$ is a unit, $EU = UE$ (equivalently $EA = AE$), and $AE$ is 
nilpotent. Since $EA = AE$ and $E^2 = E$, we have $EAE = AE$; hence $EAE$ is 
nilpotent. By \cite[Proposition~2.5]{diesl2013nil}, $A$ is strongly $\pi$-regular. 
Since $A$ was arbitrary, $M_n(\mathbb{R})$ is strongly $\pi$-regular.
\end{proof}

\begin{theorem}\label{thm:strpireg}
The matrix ring over dual numbers $M_n(\mathbb{D})$ is strongly $\pi$-regular.
\end{theorem}

\begin{proof}
Let $\mathcal{M} = A + B\epsilon$ be any dual matrix. By 
Lemma~\ref{thm:realstronglyclean}, the real part $A$ can be decomposed as 
$A = E + U$, where $E$ is an idempotent, $U$ is a unit, $EU = UE$ (equivalently 
$EA = AE$), and $AE$ is nilpotent in $M_n(\mathbb{R})$.

Our aim is to show that $\mathcal{M} = \mathcal{E} + \mathcal{U}$, where 
$\mathcal{E}$ is an idempotent, $\mathcal{U}$ is a unit, 
$\mathcal{E}\mathcal{M} = \mathcal{M}\mathcal{E}$, and $\mathcal{M}\mathcal{E}$ is 
nilpotent in $M_n(\mathbb{D})$. As in the real case, for such $\mathcal{E}$ we have 
$\mathcal{E}\mathcal{M}\mathcal{E} = \mathcal{M}\mathcal{E}$, so these conditions 
imply that $\mathcal{M}$ is strongly $\pi$-regular by 
\cite[Proposition~2.5]{diesl2013nil}.

Let $\mathcal{E} = E + X\epsilon$. Recall Theorem~\ref{thm:idempotent}: 
$\mathcal{E}$ is idempotent if and only if $E^2 = E$ and
\begin{equation}\label{eq11:X}
X = EX + XE.
\end{equation}
By Remark~\ref{remark:off-diagonal}, $X$ is off-diagonal,
$$X = \begin{pmatrix} 0 & X_{12} \\ X_{21} & 0 \end{pmatrix}.$$

Now consider commutativity:
\begin{align*}
\mathcal{E}\mathcal{M} &= (E + X\epsilon)(A + B\epsilon) = EA + (EB + XA)\epsilon,\\
\mathcal{M}\mathcal{E} &= (A + B\epsilon)(E + X\epsilon) = AE + (BE + AX)\epsilon.
\end{align*}
Since $EA = AE$, we have $\mathcal{E}\mathcal{M} = \mathcal{M}\mathcal{E}$ if and 
only if
\begin{equation}\label{eq22:AX-XA}
AX - XA = EB - BE.
\end{equation}
We need to show that there exists $X$ satisfying 
Equations~\eqref{eq11:X} and~\eqref{eq22:AX-XA}. By 
Lemma~\ref{thm:realstronglyclean}, in a basis adapted to 
$\mathbb{R}^n = \Ima(E) \oplus \ker(E)$ we have
$$E = \begin{pmatrix} I_r & 0 \\ 0 & 0 \end{pmatrix}, \qquad 
A = \begin{pmatrix} A_{11} & 0 \\ 0 & A_{22} \end{pmatrix},$$
where $A_{11}$ and $A_{22}$ have no eigenvalues in common.

Let $B = \begin{pmatrix} B_{11} & B_{12} \\ B_{21} & B_{22} \end{pmatrix}$. 
Substituting $E$, $A$, $B$, and $X$ into Equation~\eqref{eq22:AX-XA} yields
$$\begin{pmatrix} 0 & A_{11}X_{12} - X_{12}A_{22} \\ 
A_{22}X_{21} - X_{21}A_{11} & 0 \end{pmatrix}
= \begin{pmatrix} 0 & B_{12} \\ -B_{21} & 0 \end{pmatrix}.$$
We have
\begin{align}
A_{11}X_{12} - X_{12}A_{22} &= B_{12}, \label{eq:sylvester11}\\
A_{22}X_{21} - X_{21}A_{11} &= -B_{21}. \label{eq:sylvester22}
\end{align}
By Lemma~\ref{lemma:Sylvester}, Equations~\eqref{eq:sylvester11} 
and~\eqref{eq:sylvester22} yield unique matrix solutions $X_{12}$ and $X_{21}$. 
Since this solution $X$ is off-diagonal in the chosen basis, we have $EX + XE = X$, 
so Equation~\eqref{eq11:X} holds. Hence $\mathcal{E} = E + X\epsilon$ is idempotent 
by Theorem~\ref{thm:idempotent}, and 
$\mathcal{E}\mathcal{M} = \mathcal{M}\mathcal{E}$ by Equation~\eqref{eq22:AX-XA}.

Set $\mathcal{U} = \mathcal{M} - \mathcal{E}$. Then $\mathcal{U} =(A+B\epsilon)-(E+X\epsilon)= U + (B - X)\epsilon$. By 
Theorem~\ref{thm:invertible}, $\mathcal{U}$ is invertible, so it is a unit.

Finally, we have $\mathcal{M}\mathcal{E} = AE + (BE + AX)\epsilon$ and $AE$ is 
nilpotent. By Theorem~\ref{thm:nilpotent}, $\mathcal{M}\mathcal{E}$ is nilpotent. 
Hence $\mathcal{M}$ is strongly $\pi$-regular. Since every element 
of $M_n(\mathbb{D})$ is strongly $\pi$-regular, $M_n(\mathbb{D})$ is strongly $\pi$-regular.
\end{proof}

\begin{proposition}\label{prop:piregular}
The matrix ring  over dual numbers $ M_n(\mathbb{D})$ is $\pi$-regular.
\end{proposition}	
\begin{proof} By Theorem  \ref{thm:strpireg},  $ M_n(\mathbb{D}) $ is strongly $\pi$-regular. It is well known that any strongly $\pi$-regular ring is necessarily $\pi$-regular \cite[Corollary to Theorem 3]{Azumaya1954}. Thus $ M_n(\mathbb{D}) $ is $\pi$-regular.
\end{proof}

\section{The Clean Properties}

Now, we investigate clean-type properties of $M_n(\mathbb{D})$.

\subsection{Clean and Strongly Clean Properties}

\indent Recall that an element $a$ in a ring $R$ is called \textit{clean} if $ a = e + u $ where $e^2 = e \in R$ and $u$ is a unit in $R$ \cite{Nicholson1977}. A ring $R$ is called a \textit{clean ring} if every element in $R$ is clean. Furthermore, if these elements commute (i.e., $eu = ue$, or equivalently $ea = ae$), the element $a$ is called \textit{strongly clean}, and the ring $R$ is said to be a \textit{strongly clean ring} if all of its elements are strongly clean.

\begin{theorem}\label{theorem:clean}
	The matrix ring  over dual numbers $M_n(\mathbb{D})$ is a clean ring.
\end{theorem}

\begin{proof}
	Let $\mathcal{M}=A+B\epsilon$ be any dual matrix. Since $M_n(\mathbb{R})$ is known to be a clean ring  \cite[Corollary~6]{CamilloYu1994}, the real part $A$ can be decomposed as $A=E+U$, where $E$ is an idempotent and $U$ is a unit in $M_n(\mathbb{R})$.
	
	Consider the decomposition $\mathcal{M} = \mathcal{E} + \mathcal{U}$, where 
$\mathcal{E} = E$ and $\mathcal{U} = U + B\epsilon$. Since $E$ is a real idempotent, $\mathcal{E}$ is clearly idempotent in $M_n(\mathbb{D})$. Since $U$ is a unit, by Theorem \ref{thm:invertible}, $\mathcal{U}$ is also a unit in $M_n(\mathbb{D})$. Thus, every matrix in $M_n(\mathbb{D})$ has a clean decomposition.
\end{proof}

\begin{theorem} \label{theorem:stronglyclean}
The matrix ring over dual numbers $M_n(\mathbb{D})$ is strongly clean.
\end{theorem}

\begin{proof}
Let $\mathcal{M} = A + B\epsilon$ be any dual matrix. In the proof of 
Theorem~\ref{thm:strpireg}, $\mathcal{M}$ was decomposed as 
$\mathcal{M} = \mathcal{E} + \mathcal{U}$, where $\mathcal{E}$ is an idempotent, 
$\mathcal{U}$ is a unit, and $\mathcal{E}\mathcal{M} = \mathcal{M}\mathcal{E}$ (equivalently 
$\mathcal{E}\mathcal{U}= \mathcal{U}\mathcal{E}$). 
This is precisely a strongly clean decomposition of $\mathcal{M}$. Hence every 
element of $M_n(\mathbb{D})$ is strongly clean, so $M_n(\mathbb{D})$ is strongly 
clean.
\end{proof}

Note that we can also obtain the fact that Theorem \ref{thm:strpireg} implies Theorem \ref{theorem:stronglyclean}, since every strongly $\pi$-regular ring is strongly clean \cite[Proposition~2.6]{burgess88}.
\subsection{Nil-Clean and Strongly Nil-Clean Properties}

Recall that an element $a$ in a ring $R$ is called \textit{nil-clean} if it can be written as $a = e + n$, where $e$ is an idempotent and $n$ is a nilpotent element. A ring $R$ is called a \textit{nil-clean ring} if every element in $R$ is nil-clean. 
	
Furthermore, if these elements commute (i.e., $en = ne$, or equivalently $ea = ae$), the element $a$ is called \textit{strongly nil-clean}, and $R$ is said to be a \textit{strongly nil-clean ring} if all of its elements are strongly nil-clean.

\begin{theorem}\label{thm:nil-cleanChar}
    Let $\mathcal{M}=A+B\epsilon \in M_n(\mathbb{D})$. Then $\mathcal{M}$ is nil-clean if and only if $A$ is nil-clean in $M_n(\mathbb{R})$.
\end{theorem}

\begin{proof}
    ($\Rightarrow$) Assume $\mathcal{M}=\mathcal{E}+\mathcal{N}$ is a nil-clean decomposition. By Lemma \ref{lemma:proj}, applying $\pi$, we get $A=\pi(\mathcal{E})+\pi(\mathcal{N})$. Since homomorphisms preserve algebraic properties, $\pi(\mathcal{E})$ is idempotent and $\pi(\mathcal{N})$ is nilpotent. Thus $A$ is nil-clean.

    ($\Leftarrow$) Assume $A$ is nil-clean, i.e., $A=E+N$. Then
    $$ \mathcal{M}=E+(N+B\epsilon) $$
    Let $\mathcal{E}=E$ and $\mathcal{N}=N+B\epsilon$. Clearly, $\mathcal{E}$ is idempotent. For $\mathcal{N}$, since $N$ is nilpotent, by Theorem \ref{thm:nilpotent}, $N+B\epsilon$ is also nilpotent.
    Thus, $\mathcal{M}=\mathcal{E}+\mathcal{N}$ is a nil-clean decomposition.
\end{proof}

\begin{proposition}
  The matrix ring over dual numbers $M_n(\mathbb{D})$ is not nil-clean.
\end{proposition}
\begin{proof} Assume to the contrary $M_n(\mathbb{D})$ is nil-clean. Then by Theorem \ref{thm:nil-cleanChar} every real matrix $A\in M_n(\mathbb{R})$ would be nil-clean. Consider $A= 2I_n$. Since $A$ is nil-clean, there exists an idempotent matrix $E$ and nilpotent matrix $N$ such that $A=E+N$. Then $$N=2I-E.$$ The eigenvalues of the idempotent $E$ are $\lambda_E \in \{0, 1\}$. Then the eigenvalues of $N$ are $\lambda_N \in \{2, 1\}$. However, a nonzero nilpotent matrix must have all zero eigenvalues which contradicts what we have obtained.

\end{proof}

\begin{proposition}
 If $\mathbb{D}={\mathbb{F}_2 [\epsilon]} / {\langle {\epsilon}^2 \rangle}$, then $M_n(\mathbb{D})$ is nil-clean.
\end{proposition}
\begin{proof} By \cite[Theorem~3]{breaz2013} we know that  the matrix ring $M_n(\mathbb{F}_2)$ is nil-clean. In this case, by Theorem \ref{thm:nil-cleanChar} every $\mathcal{M}=A+B\epsilon$ in $M_n(\mathbb{D})$ is nil-clean, by setting $\mathcal{E}=E$ and $\mathcal{N}=N+B\epsilon$.
\end{proof}

\begin{theorem} Let $\mathcal{M}=A+B\epsilon \in M_n(\mathbb{D})$. Then $\mathcal{M}$ is strongly nil-clean if and only if $A$ is strongly nil-clean in $M_n(\mathbb{R})$.
\end{theorem} 
\begin{proof}

    ($\Rightarrow$)  Assume $\mathcal{M} = \mathcal{E} + \mathcal{N}$ is a strongly 
nil-clean decomposition, that is, $\mathcal{E}$ is an idempotent, $\mathcal{N}$ is 
nilpotent, and $\mathcal{E}\mathcal{N} = \mathcal{N}\mathcal{E}$. By 
Lemma~\ref{lemma:proj}, applying $\pi$ we get 
$A = \pi(\mathcal{E}) + \pi(\mathcal{N}) = E + N$, and 
$\pi(\mathcal{E}\mathcal{N}) = \pi(\mathcal{N}\mathcal{E})$ implies $EN = NE$. 
Since $\pi$ is a ring homomorphism, $E$ is an idempotent and $N$ is nilpotent in 
$M_n(\mathbb{R})$. Thus $A$ is strongly nil-clean.

($\Leftarrow$) Assume $A$ is strongly nil-clean, i.e., $A = E + N$, where 
$E^2 = E$, $N$ is nilpotent, and $EN = NE$ (equivalently $EA = AE$).

Let $\mathcal{E} = E + X\epsilon$. Recall Theorem~\ref{thm:idempotent}: 
$\mathcal{E}$ is idempotent if and only if $E^2 = E$ and
\begin{equation}\label{eq:X-nilclean}
X = EX + XE.
\end{equation}
By Remark~\ref{remark:off-diagonal}, $X$ is off-diagonal $X = \begin{pmatrix} 0 & X_{12} \\ X_{21} & 0 \end{pmatrix}$ in a basis where 
$E = \begin{pmatrix} I_r & 0 \\ 0 & 0 \end{pmatrix}$.

For commutativity, as in the proof of Theorem~\ref{thm:strpireg}, 
$\mathcal{E}\mathcal{M} = \mathcal{M}\mathcal{E}$ (equivalently 
$\mathcal{E}\mathcal{N} = \mathcal{N}\mathcal{E}$) holds if and only if
\begin{equation}\label{eq:AX-XA=EB-BE}
AX - XA = EB - BE.
\end{equation}
We need to show that there exists $X$ satisfying 
Equations~\eqref{eq:X-nilclean} and~\eqref{eq:AX-XA=EB-BE}.

Since $AE = EA$ and $N = A - E$, the matrices $A$ and $N$ are block diagonal in 
the chosen basis:
$$A = \begin{pmatrix} A_{11} & 0 \\ 0 & A_{22} \end{pmatrix}, \qquad 
N = \begin{pmatrix} N_{11} & 0 \\ 0 & N_{22} \end{pmatrix},$$
where $A_{11} = I_r + N_{11}$ and $A_{22} = N_{22}$. Since $N$ is nilpotent, 
$N_{11}$ and $N_{22}$ are nilpotent, so their only eigenvalue is zero. Hence the 
only eigenvalue of $A_{11}$ is $1$, while the only eigenvalue of $A_{22}$ is $0$. Thus $A_{11}$ and $A_{22}$ have no eigenvalues in common.

Let $B = \begin{pmatrix} B_{11} & B_{12} \\ B_{21} & B_{22} \end{pmatrix}$. 
Substituting $E$, $A$, $B$, and $X$ into Equation~\eqref{eq:AX-XA=EB-BE}, we obtain
\begin{align}
A_{11}X_{12} - X_{12}A_{22} &= B_{12}, \label{eq:sylvester3}\\
A_{22}X_{21} - X_{21}A_{11} &= -B_{21}. \label{eq:sylvester4}
\end{align}
By Lemma~\ref{lemma:Sylvester}, Equations~\eqref{eq:sylvester3} 
and~\eqref{eq:sylvester4} yield unique matrix solutions $X_{12}$ and $X_{21}$. 
Since this solution $X$ is off-diagonal in the chosen basis, we have 
$EX + XE = X$, so Equation~\eqref{eq:X-nilclean} holds. Hence 
$\mathcal{E} = E + X\epsilon$ is idempotent by Theorem~\ref{thm:idempotent}, and 
$\mathcal{E}\mathcal{M} = \mathcal{M}\mathcal{E}$ by 
Equation~\eqref{eq:AX-XA=EB-BE}.

Set $\mathcal{N} = \mathcal{M} - \mathcal{E} = N + (B - X)\epsilon$. Since $N$ is 
nilpotent, $\mathcal{N}$ is nilpotent in $M_n(\mathbb{D})$ by 
Theorem~\ref{thm:nilpotent}.

Consequently, $\mathcal{M} = \mathcal{E} + \mathcal{N}$, where $\mathcal{E}$ is an 
idempotent, $\mathcal{N}$ is nilpotent, and 
$\mathcal{E}\mathcal{M} = \mathcal{M}\mathcal{E}$. Hence $\mathcal{M}$ is strongly 
nil-clean.
\end{proof}
\begin{proposition}
  The matrix ring over dual numbers $M_n(\mathbb{D})$ is not strongly nil-clean.
\end{proposition}
\begin{proof} Clearly, every strongly nil-clean ring is nil-clean. Since $M_n(\mathbb{D})$ is not nil-clean, it is not strongly nil-clean.
\end{proof}

\subsection{$r$-Clean and Strongly $r$-Clean Properties}
Recall that an element $a$ in a ring is called \textit{r-clean} (or \textit{regular-clean}) if 	$a=r+e$	where $r$ is von Neumann regular and $e$ is idempotent \cite{Ashrafi2011rcleanR}.

\begin{corollary}
	The matrix ring over dual numbers $M_n(\mathbb{D})$ is an $r$-clean ring.
\end{corollary}
\begin{proof}  
By Theorem \ref{theorem:clean}, $M_n(\mathbb{D})$ is clean. Since every unit is von Neumann regular, it is immediate that any clean ring is $r$-clean. Thus $M_n(\mathbb{D})$ is an $r$-clean ring.
\end{proof}
Recall from the Introduction that an element $a$ in a ring is called \textit{strongly r-clean} if $a=r+e$	and $re=er$, where $r$ is von Neumann regular and $e$ is idempotent \cite{sharma2018stronglyrclean}.

\begin{corollary}
	The matrix ring over dual numbers $M_n(\mathbb{D})$ is a strongly $r$-clean ring.
\end{corollary}
\begin{proof}
By Theorem~\ref{theorem:stronglyclean}, $M_n(\mathbb{D})$ is strongly clean. Since 
every unit is von Neumann regular, every strongly clean element is strongly 
$r$-clean. Thus $M_n(\mathbb{D})$ is a strongly $r$-clean ring.
\end{proof}

\subsection{The Nil Regular $(\mathrm{NR})$-Clean Property}
An element $a$ in a ring is called \textit{$\mathrm{NR}$-clean} if it can be written as the sum of a regular element and a nilpotent
\[a=r+n.\]  Furthermore, if these elements commute  $rn = nr$,  the element $a$ is called \textit{strongly $\mathrm{NR}$-clean}, and the ring $R$ is said to be a \textit{strongly $\mathrm{NR}$-clean ring} if all of its elements are strongly $\mathrm{NR}$-clean \cite{Khashan2016}, \cite{Ibraheem2021}.

\begin{theorem}
	The matrix ring over dual numbers $M_n(\mathbb{D})$ is an $\mathrm{NR}$-clean ring.
\end{theorem}

\begin{proof}
Let $\mathcal{M}=A+B\epsilon$ be any dual matrix. Since $M_n(\mathbb{R})$ is a regular ring, the real part $A$ is also regular in $M_n(\mathbb{D})$. Clearly, $(B\epsilon)^2=0$, so $B\epsilon$ is nilpotent. Hence, every element $\mathcal{M} \in M_n(\mathbb{D})$ can be decomposed as the sum of a regular element and a nilpotent element.
\end{proof}

\begin{proposition}\label{prop:uniquelyNR} The matrix ring over dual numbers $M_n(\mathbb{D})$ is not uniquely $\mathrm{NR}$-clean. 
\end{proposition}
\begin{proof} Take $ \mathcal{M}=I+B\epsilon$ with $B\neq 0$. Consider 
$ \mathcal{M}=(I+B\epsilon)+0$, where  $I+B\epsilon$ is regular with an inner inverse $I-B\epsilon$, and clearly $0$ is nilpotent. On the other hand, 
$ \mathcal{M}=I+B\epsilon$, where $I$ is regular, and $B\epsilon$ is nilpotent. These two decompositions have distinct regular parts, so $M_n(\mathbb{D})$ is not uniquely $\mathrm{NR}$-clean.
\end{proof}
\begin{remark} By \cite[Theorem~4]{Khashan2016} a ring $R$ is a uniquely $\mathrm{NR}$-clean ring if and only if $R$ is a strongly regular ring. Since $M_n(\mathbb{D})$ is not strongly regular by Corollary \ref{corollary:notstrregular}, this provides an alternative proof that $M_n(\mathbb{D})$ is not uniquely $\mathrm{NR}$-clean. Note also that the two decompositions in Proposition~\ref{prop:uniquelyNR} 
have commuting components, so the strongly $\mathrm{NR}$-clean decomposition of 
a dual matrix is not unique in general. On the other hand, the strongly $\pi$-regular decomposition is always unique 
\cite[Proposition~2.6]{diesl2013nil}.
\end{remark}
\begin{theorem} The matrix ring over dual numbers $M_n(\mathbb{D})$ is a strongly $\mathrm{NR}$-clean ring.
\end{theorem}

\begin{proof} Let $\mathcal{M}\in M_n(\mathbb{D})$. By Theorem \ref{thm:strpireg}, $M_n(\mathbb{D})$ is strongly $\pi$-regular so that  $\mathcal{M}$ can be decomposed as $\mathcal{M}=\mathcal{E}+ \mathcal{U}$, where  $\mathcal{E}$ is an idempotent, $ \mathcal{U}$ is a unit, $ \mathcal{E}\mathcal{M}=\mathcal{M} \mathcal{E}$ and $\mathcal{M}\mathcal{E}$ is nilpotent in $M_n(\mathbb{D})$. Set 
$$ \mathcal{R}= \mathcal{M}(I-\mathcal{E}) \mbox{ and } \mathcal{N}= \mathcal{M}\mathcal{E}$$ 
Then $\mathcal{M}=\mathcal{R}+ \mathcal{N}$. Since $ \mathcal{E}\mathcal{M}=\mathcal{M} \mathcal{E}$, and $\mathcal{E}^2=\mathcal{E}$, the elements $\mathcal{R}$ and $\mathcal{N}$ commute
$$\mathcal{R} \mathcal{N}= \mathcal{M}(I-\mathcal{E})\mathcal{M}\mathcal{E}=\mathcal{M}^2\mathcal{E}-\mathcal{M}\mathcal{E}\mathcal{M}\mathcal{E}=0, $$ 
$$\mathcal{N} \mathcal{R}= \mathcal{M}\mathcal{E}\mathcal{M}(I-\mathcal{E})=\mathcal{M}\mathcal{E}\mathcal{M}-\mathcal{M}\mathcal{E}\mathcal{M}\mathcal{E}=0.$$
By assumption $\mathcal{N}=\mathcal{M}\mathcal{E}$ is nilpotent. For regularity of $\mathcal{R}$, since $\mathcal{E}^2=\mathcal{E}$,
$$\mathcal{R}=\mathcal{M}(I-\mathcal{E})=(\mathcal{E}+ \mathcal{U})(I-\mathcal{E})=\mathcal{E}(I-\mathcal{E})+\mathcal{U}(I-\mathcal{E})=\mathcal{U}(I-\mathcal{E}).$$
Note that $(I-\mathcal{E})$ is an idempotent. Let $\mathcal{U}^{-1}$ be the inverse of unit $\mathcal{U}$. Then
$$\mathcal{R}\mathcal{U}^{-1}\mathcal{R}=\mathcal{U}(I-\mathcal{E})\mathcal{U}^{-1}\mathcal{U}(I-\mathcal{E})=
\mathcal{U}(I-\mathcal{E})^2=\mathcal{U}(I-\mathcal{E})=\mathcal{R}.$$ Hence $\mathcal{R}$ is regular with inner inverse $\mathcal{U}^{-1}$. Consequently, $M_n(\mathbb{D})$ is  strongly $\mathrm{NR}$-clean.
\end{proof}

\section{Conclusion}
In this paper, we characterized the fundamental regularity and clean properties of the matrix ring over dual numbers $M_n(\mathbb{D})$. We distinguished between global and local behaviors. Some properties hold for the whole ring $M_n(\mathbb{D})$, while others depend on the choice of an individual matrix.

\begin{itemize}
	\item \textbf{Global Ring Properties (True for all of $M_n(\mathbb{D})$):}
	\begin{itemize}
        \item \textbf{$\pi$-Regular and Strongly $\pi$-Regular:} $M_n(\mathbb{D})$ inherits strong $\pi$-regularity from $M_n(\mathbb{R})$, and consequently is $\pi$-regular.
		\item \textbf{Clean and Strongly Clean:} $M_n(\mathbb{D})$ is both clean and strongly clean. The off-diagonal blocks are uniquely determined by solving Sylvester matrix equations.
		
		\item \textbf{$r$-Clean and Strongly $r$-Clean:}  Since $M_n(\mathbb{D})$ is strongly clean, it is both strongly $r$-clean and $r$-clean ring.
        \item \textbf{$\mathrm{NR}$-Clean and Strongly $\mathrm{NR}$-Clean:} Since $M_n(\mathbb{R})$ is regular and $B\epsilon$ is nilpotent, $M_n(\mathbb{D})$ is $\mathrm{NR}$-Clean. Moreover, using the strongly $\pi$-regular structure of $M_n(\mathbb{D})$, every element admits a commuting decomposition into a regular and a nilpotent part, hence $M_n(\mathbb{D})$ is strongly $\mathrm{NR}$-Clean.

	\end{itemize}
	
	\item \textbf{Failed Global Properties (Not true for all of $M_n(\mathbb{D})$):}
	\begin{itemize}
		\item \textbf{von Neumann Regular and Strongly Regular:} $M_n(\mathbb{D})$ is not von Neumann regular because an arbitrary dual matrix $\mathcal{M} = A + B\epsilon$ requires the compatibility condition $(I-AA^{-})B(I-A^{-}A)=0$, which fails in general. For 
instance, for $\mathcal{M} = B\epsilon$ with $B \neq 0$. Consequently,  $M_n(\mathbb{D})$ is not strongly regular either. 
		\item \textbf{Nil-Clean and Strongly Nil-Clean:} $M_n(\mathbb{D})$ is not nil-clean in general (it is nil-clean only over the field $\mathbb{F}_2$). Consequently, $M_n(\mathbb{D})$ is not strongly nil-clean either.
	\end{itemize}
	
	\item \textbf{Local Matrix Properties (For a dual matrix $\mathcal{M}=A+B\epsilon$):}
	\begin{itemize}
		\item $\mathcal{M}$ is invertible $\iff$ $A$ is invertible.
		\item $\mathcal{M}$ is idempotent $\iff$ $A$ is idempotent and $AB+BA=B$.
		\item $\mathcal{M}$ is nilpotent $\iff$ $A$ is nilpotent.
        \item $\mathcal{M}$ is regular $\iff$ $A$ is regular and $(I-AA^{-})B(I-A^{-}A)=0$ for some inner inverse $A^{-}$ of $A$.
        \item $\mathcal{M}$ is $\pi$-regular  $\iff$ there exists $k\geq 1$ such that $A^k$ is regular and  $(I-A^k(A^k)^{-})C_k(I-(A^k)^{-}A^k)=0$ for some inner inverse $(A^k)^{-}$, where $ C_k:=\sum_{i=0}^{k-1}A^iBA^{(k-1)-i}$.
		\item $\mathcal{M}$ is nil-clean (resp. strongly nil-clean) $\iff$ $A$ is nil-clean (resp. strongly nil-clean) in $M_n(\mathbb{R})$.
        
	\end{itemize}
\end{itemize}

\section{Acknowledgements}
The author thanks Nesl\.{ı}han Ay\c{s}en \"{O}zbay for suggesting the study of matrices over dual numbers.

\bigskip

\noindent\textsc{BERR\.{I}N  \c SENT\" URK} ($^*$)\\
\noindent Department of Mathematics, TED University, 06420, Ankara, Turkey\\
\noindent\textit{e-mail address}: \texttt{berrin.senturk@tedu.edu.tr}\\

\end{document}